\documentclass[a4paper]{article}

\usepackage[T1]{fontenc}
\usepackage[english]{babel}
\usepackage[latin1]{inputenc}

\usepackage{amsfonts,amsmath,amssymb,amsthm,amscd,latexsym,geometry}
\usepackage{graphicx}
\usepackage{color}
\usepackage{fancyhdr}
\usepackage{tikz}

\usepackage[pdftex=true,hyperindex=true,colorlinks=true,linkcolor=blue,bookmarks
=true]{hyperref}
\usepackage{lmodern}
\usepackage{eulervm}
\newcommand{\R}{\mathbb{R}}
\newcommand{\PP}{\mathbb{P}}

\newcommand{\e}{\varepsilon}
\newcommand{\D}{\Delta}
\newcommand{\N}{\mathbb{N}}

\newcommand{\dd}{{\rm d}}
\newcommand{\xen}{X^{\varepsilon,N}}
\renewcommand{\yen}{Y^{\varepsilon,N}}
\newcommand{\xbn}{\overline{X}^{N}}

\newcommand{\ddt}{\frac{\rm d}{{\rm d}t}}

\newtheorem{assumption}{Assumption}
\newtheorem{proposition}{Proposition}

\newtheorem{theorem}{Theorem}
\newtheorem{corollary}{Corollary}

\title{Spatio-temporal averaging for a class of hybrid systems and application to conductance-based neuron models}

\author{Alexandre Genadot\footnote{Institut de Math\'ematiques de Bordeaux UMR 5251, Universit\'e de Bordeaux, 351, cours de la Lib\'eration, 33 405 Talence, France. Electronic address: Alexandre.Genadot@math.u-bordeaux1.fr}}

\date{}

\begin{document}

\maketitle

\begin{abstract}
We obtain a limit theorem endowed with quantitative estimates for a general class of infinite dimensional hybrid processes with intrinsically two different time scales and including a population. As an application, we consider a large class of conductance-based neuron models describing the nerve impulse propagation along a neural cell at the scales of ion channels.
\end{abstract}

\section{Introduction} 

This article aims to study a class of infinite dimensional hybrid processes $(Z^{\e,N}_t,t\geq0)$ with intrinsically two different time scales, this fact being emphasized by the presence of the parameter $\e$, and including a population whose size's parameter is $N$. We show the convergence of these models to a limit model when $\e$ and $N$ go to zero and infinity, respectively, and with a certain relative speed. This corresponds to consider an infinite population whose fast component is infinitely accelerated. Our main motivation for studying such a situation comes from mathematical neurosciences and more specifically from conductance-based neuron models describing the generation and propagation of a nerve impulse along a neural cell at the scale of ion channels. Our main result stands in this case a sufficient condition under which the used of reduced (or averaged) conductance-based neuron models is justified.\\

Hybrid systems, combining discrete and continuous dynamics, have been actively studied in applied mathematics. As highlighted in the recent review \cite{HP14}, they arise in a great variety of applications from bio-physiology through, as already mentioned, the modeling of excitable cells (neural or cardiac cells) \cite{SB09}, or the description of molecular motors \cite{H01, F10}, to the development of cyber-physical systems \cite{BBCP13,DLV12}, that is systems that interact tightly with the physical world and human operators such as in air traffic control \cite{RG04}, to mention just a few.\\

The hybrid processes that we consider can be described as follows. They have two distinct components, $Z^{\e,N}=(X^{\e,N},Y^{\e,N})$, one continuous $(X^{\e,N})$ and the other one of pure jumps $(Y^{\e,N})$. The continuous component is often referred as the macroscopic one in neural modeling whereas the term microscopic stands for the component of pure jumps. Between two jumps of this discrete component, the continuous component evolves according to an abstract evolution equation. When the jump component updates its current value, the parameters of the evolution equation are also updated; and so on. The evolution that we just described is that of a piecewise deterministic process. We will work more specifically in the context of Piecewise Deterministic Markov Process (PDMP). The study of such hybrid processes was carried out comprehensively in the works \cite{Da93,Da84} concerning the finite dimension, that is when the evolution equation for the continuous component is an ordinary differential equation (ODE). In the case where the ODE becomes a partial differential equation (PDE), and more generally in our setting an abstract evolution equation, the theory has been established in \cite{BR11}.\\

In recent years, these processes have been widely studied in terms of averaging and law of large numbers. That is, we consider either that the microscopic component evolves faster than the macroscopic component, or that the size of the internal population of the model goes to infinity. On the one hand, regarding the class of processes that we consider, limit theorems when $N$ goes to infinity such as the law of large numbers and the central limit theorem have been obtained in \cite{Au08,BR13,RTW12,TR13}. On the other hand, the study of the averaging and the associated fluctuations when $\e$ goes to zero was carried out in \cite{GT12,GT14}. Our goal in the present paper is to reconcile the two approaches. Considering a general class of PDMP's with intrinsically two time-scale and a size parameter, we obtain a result of law of large number type with an explicit rate of convergence, for the joint convergence in $(\e,N)$. The mathematical tools we use in the proof belong to the singular perturbation theory throughout the study of a Poisson equation, and to general probabilities through the study of a martingale problem. Regarding the results previously obtained for this kind of models, and especially the cited averaging results, the main novelty comes from obtaining quantitative bounds for convergence thanks to the use of exponential inequalities for martingales.  These quantitative estimates are required to infer a criterion on the relative speed between the time-scale parameter $\e$ and the population size $N$, allowing joint convergence.\\

Our main motivation is the study of hybrid models describing the generation and propagation of a nerve impulse along a neuron. Indeed, this action potential evolves along the nerve fiber (or axon) according to a partial differential equations whose parameters are updated based on the current state of the ion channels present all along the axon. These ion channels allow ion exchange between the inside and outside of the cell. They are located all along the axon in discrete sites, forming a population of size $N$. Their mechanism of opening and closing, depending on the local action potential of the cell membrane,  is responsible for the generation and propagation of nerve impulses (on these questions, see the comprehensive book \cite{Hi01}). It has been shown that certain ionic channels evolve more rapidly than others, and it is now common to include a small parameter $\e$ in the model to account for this characteristic. Therefore, these neuron models fall naturally into the class of PDMP's that we propose to study. One challenge in the modeling of phenomena including multiple time and space scale, is to obtained reduced models easier to handle, both analytically and numerically. The reduced models are then used as appropriate approximate models. It is then quite important to know when such approximations are valid. Through our study, we obtain a sufficient condition under which the neuron model with the two parameters $\e$ and $N$ converges to a limit as these parameters go respectively to zero and infinity.  This condition involves the relative speed of convergence between the two different scales $\e$ and $N$. Heuristically, the result is as follows: somehow, it is enough that the ion channels evolve faster than the population grows. \\

The plan of the paper is the following. In Section \ref{sec_res}, we define the model of interest and state our main results. Section \ref{sec_appli} is devoted to the application of our results to conductance-based neuron models, our first motivation. The proof of our main results is presented in Section \ref{sec_proof}. 

\section{Presentation of the model and main results}\label{sec_res}

\noindent Let $H$ be a separable Hilbert space and $H^*$ its dual. We denote by $(\cdot,\cdot)$ the inner product on $H$ and by $\|\cdot\|$ the associated norm. The duality bracket between $H$ and $H^*$ is denoted by $<\cdot,\cdot>$. We define a sequence $(E_N)_{N\in\N}$ of finite spaces with increasing cardinality such that $E_N\subset E_{N+1}$ according to the canonical injection\footnote{The canonical injection is given by $\iota : i\in E_N\mapsto (i,\bold{0})\in E_{N+1}$ where $\bold{0}$ is the zero of $\R^{|E|_{N+1}-|E|_N}$.}. For any $x\in H$ we define an intensity matrix $Q^N(x)=(Q^N_{y_1y_2}(x))_{(y_1,y_2)\in E_N\times E_N}$.
\begin{assumption}\label{ass_Q}
We assume that for each $x\in H$, there is a unique quasi-stationary probability measure $\mu_N(x)$ on $E_N$ associated to $Q^N(x)$ such that
\[
\mu_N(x)Q^N(x)=0.
\]
Moreover, we assume that the jump rates are all uniformly bounded:
\begin{equation}
\exists [Q]\in\R_+\,\,\forall N\in\N\,\,\forall(y_1,y_2,x)\in E_N\times E_N\times H\quad |Q^N_{y_1y_2}(x)|\leq [Q]. 
\end{equation}
\end{assumption}
\noindent The term quasi-stationary refers to the fact that the stationary measure $\mu_N(x)$ actually depends on the external variable $x$. We will use this intensity matrix to define the jumping component of the studied process. For example, in the neuron model described in Section \ref{sec_mod}, the intensity matrix $Q^N(x)$ give the rate at which ionic channels change of state (roughly speaking, open and close) at a given potential $x$. The boundedness property of these rates as well as the existence of a quasi-stationary distribution are then naturally satisfied for such models.  Before stating the equations of our process, we need to define a linear self-adjoint operator $A$ on $H$ with domain $\mathcal{D}(A)$ continuously and densely embedded in $H$. For any $N\in\N$, we also define a reaction term $F^N$ on $H\times E_N$. Our assumptions on $A$ and $F^N$ are gathered below.
\begin{assumption}\label{ass_mac}
The linear operator $A$, with domain $\mathcal{D}(A)$ continuously and densely embedded in $H$, is dissipative in the sense that
\begin{equation}\label{A_d}
\exists [A]_d\in \R^*_+\,\,\forall x\in H\quad <Ax,x>\leq -[A]_d\|x\|^2.
\end{equation}
For any $N\in\N$, the reaction term $F^N$ defined on $H$ satisfies the growth condition
\begin{equation}\label{F_d}
\exists [F]_d\in \R^*_+\,\,\forall N\in\N\,\,\forall (x_1,x_2,y)\in H\times H\times E_N\quad (F^N(x_1,y)-F^N(x_2,y),x_1-x_2)\leq [F]_d\|x_1-x_2\|^2.
\end{equation}
Moreover the value at zero is bounded such that
\begin{equation}\label{F_0}
\exists[F]_0\in\R^*_+\,\,\forall N\in\N\,\,\forall y\in E_N\quad\|F^N(0,y)\|\leq [F]_0.
\end{equation}
\end{assumption}
\noindent Let $(\e,N)\in(0,1)\times\N$. The process $(\xen,\yen)$ we consider satisfies,
\begin{itemize}
\item[a)] the following abstract evolution equation on the Hilbert space $H$ for the so-called macroscopic component $\xen$
\begin{equation}\label{X_eN}
\partial_t \xen=A\xen+F^N(\xen,\yen),
\end{equation}
\item[b)] which is fully-coupled with the following jump evolution for the so-called microscopic component $\yen$
\begin{equation}
\PP(\yen_{t+h}=y_2|\yen_t=y_1)=\frac{1}{\e}Q^N_{y_1y_2}(\xen_t)h+{\rm o}\left(\frac{h}{\e}\right),
\end{equation}
for any $y_1\neq y_2$ in $E_N$.
\end{itemize}
We endow these equations with the initial conditions $\xen_0\in H$ and $\yen_0\in E_N$ which may be random.  For the neuron model of Section \ref{sec_mod}, the operator $A$ will correspond to the Laplacian operator, which satisfies Assumption \ref{ass_mac}, taking account for the propagation of the nerve impulse along the neural cell once it is generated thanks to the reaction term $F^N$. The assumption about the growth condition of this reaction term is generally satisfied thanks to the locally Lipschitz property of this term in applications.\\
Notice that if $\xen$ is frozen to the value $\xen=x$, then $\yen$ follows the dynamic of an usual continuous time Markov chain. See \cite{YQ98} for a comprehensive study of continuous time Markov chain with fast transition rates. Under Assumptions \ref{ass_Q} and \ref{ass_mac}, we can show that for any fixed $(\e,N)$, the couple $(\xen,\yen)$ is a Hilbert-valued Piecewise Deterministic Markov Process in the sense of \cite{BR11}. We denote by $\mathcal{D}(\mathcal{A}^{\e,N})$ the domain of its extended generator. For $N\in\N$, let $\phi: H\times E_N\to\R$ be a bounded measurable function whose Fr\'echet derivative $\frac{\dd\phi}{\dd x}(x,y)\in H^*$ may be represented by an element $\phi_x(x,y)$ of $H$ for all $(x,y)\in H\times E_N$. Then $\phi$ is in $\mathcal{D}(\mathcal{A}^{\e,N})$ and in this case we have
\begin{equation}
\mathcal{A}^{\e,N}\phi(x,y)=<Ax+F^N(x,y),\phi_x(x,y)>+\frac{1}{\e}Q^N(x)\phi(x,\cdot)(y).
\end{equation}
We refer to \cite{BR11}, Theorem 4 and Section A.3 for more details about the extended generator of such processes. Our aim is to study the behavior of the process $\xen$ when $(\e,N)\to(0,\infty)$. For this purpose, we introduce the averaged version of $\xen$ when $N$ is held fixed but $\e$ goes to zero. We denote this process by $\xbn$. When $\e$ goes to zero, we expect that the process $\yen$ reaches instantaneously its stationary behavior such that $\xbn$ is solution of the abstract evolution equation:
\begin{equation}\label{X_barN}
\partial_t \xbn=A\xbn+\overline{F}^N(\xbn),
\end{equation}
where the map $\overline{F}^N$ is defined as the average of the reaction term $F^N$ against the averaging measure $\mu_N$:
\begin{equation}
\overline{F}^N(x)=\int_{E_N}F^N(x,y)\mu_N(x)(\dd y).
\end{equation}
Let us notice the trivial but important fact that $\xbn$ is a deterministic process: the stochastic jumps of $\yen$ disappeared at the limit. Such an averaging result have been obtain for example in \cite{GT12}.  Our aim is to quantify the rate of convergence of $\xen$ towards $\xbn$ in order to obtain a criterion for the joint convergence in $(\e,N)$. Our assumptions on the averaged reaction term $\overline{F}^N$ are gathered below.
\begin{assumption}\label{ass_mac2}
The reaction term $\overline{F}^N$ satisfies the growth condition:
\begin{equation}\label{bar_F_d}
\exists [\overline{F}]_d\in \R\,\,\forall (x_1,x_2,N)\in H\times H\times\N\quad (\overline{F}^N(x_1)-\overline{F}^N(x_2),x_1-x_2)\leq [\overline{F}]_d\|x_1-x_2\|^2.
\end{equation}
Moreover, it is uniformly bounded at zero: 
\begin{equation}\label{F_0}
\exists[\overline{F}]_0\in\R_+\quad\forall y\in E\quad\|\overline{F}^N(0)\|\leq [\overline{F}]_0.
\end{equation}
\end{assumption}
\noindent In applications, the growth condition (\ref{bar_F_d}) is often satisfied thanks to the fact that the quasi-stationary measure $\mu_N(x)$ satisfies in $x$ a similar growth-condition. For example, for conductance-based neuron models, for $x\in H$, $\mu_N(x)=\sum_{y\in E_N}\mu_N(x)(\{y\})\delta_y$ is such that for $y\in E_N$, the mass $\mu_N(x)(\{y\})$ is Lipschitz in $x$. Note that Assumptions \ref{ass_mac} and \ref{ass_mac2} are quite classical to ensure the existence of a unique solution to reaction-diffusion equations, see \cite{RR04}. For example, it is quite easy to deduce \textit{a priori} estimates from these assumptions. For any constant $C$, the centered closed ball with radius $C$ is denoted by $B[0,C]$. From now on, $T$ denotes a finite time horizon.
\begin{proposition}\label{prop_bound}
There exists a constant $[X]_b$ depending on $T$ but otherwise not on $\e$ and $N$ such that, $\PP$-a.s,
\[
\forall t\in[0,T]\quad \xen,\xbn\in B[0,[X]_b].
\]
\end{proposition}
\begin{proof}
$\PP$-a.s. for any $t\in[0,T]$, $\e\in(0,1)$, $N\in\N$, we have, using (\ref{A_d}) and (\ref{F_d}):
\begin{align*}
\ddt\| \xen_t\|^2&=2<A\xen_t,\xen_t>+2(F^N(\xen_t,\yen_t),\xen_t)\\
&\leq -2[A]_d\|\xen_t\|^2+2[F]_d\|\xen_t\|^2+2[F]_0\|\xen_t\|\\
&\leq (1+2[F]_d-2[A]_d)\|\xen_t\|^2+[F]^2_0.
\end{align*}
The result follows using usual comparison results. The proof is quite similar for the process $\xbn$.
\end{proof}
\noindent We now state the technical assumptions needed to prove our main result. They concern properties of the reaction term $F^N$, $\overline{F}^N$ and the operator related to the jumps $Q^N$.
\begin{assumption}\label{Assump_for_Poi}
The following assumptions hold.
\begin{itemize}
\item[a)] Let $(\alpha_N)$, $(\beta_N)$, $(\gamma_N)$ be three positive numerical sequences going to infinity with $N$. For any $N\in\N$, the map $\Phi\,:\,(x,y,z)\in H\times E_N\times \R_+\mapsto (F^N(x,y)-\overline{F}^N(x),x-z)$ satisfies:
\begin{enumerate}
\item For any $N\in\N$, the map $t\mapsto \Phi(x,y,\xbn_t)$ is in $\mathcal{D}(\mathcal{A}^{\e,N})$;
\item $\Phi$ is locally bounded, but not necessarily uniformly in $N$:
\[
\forall C>0\quad \sup_{x,z\in B[0,C],y\in E_N}|\Phi(x,y,z)|={\rm O}(\alpha_N);
\]
\item The map $\Phi$is continuously Fr\'echet differentiable with respect to $x$ and has a representation $\Phi_x$ in H which is locally bounded, but not necessarily uniformly in $N$:
\[
\forall C>0\quad \sup_{x,z\in B[0,C],y\in E_N}\|\Phi_x(x,y,z)\|={\rm O}(\beta_N);
\]
\item The derivative $\Phi_t$ of the map $t\mapsto \Phi(x,y,\xbn_t)$ with respect to $t$ exists, is continuous and locally bounded, but not necessarily uniformly in $N$:
\[
\forall C>0\quad \sup_{t\in [0,T],x\in B[0,C],y\in E_N}|\Phi_t(x,y,\xbn_t)|={\rm O}(\gamma_N);
\]
\end{enumerate}
\item[b)] For any $N\in\N$ and $C>0$, we assume that if $f^N: H\times E_N\times\R_+\to \R$ is a bounded function on $B[0,C]\times E_N\times[0,T]$, then there exists constants $C'$ and $\rho_{N,f^N}$, going to infinity with $N$, such that:
\begin{equation}\label{quatb}
\sup_{t\in[0,T],x\in B[0,C],y_1\in E_N}\left|\sum_{y\in E_N} Q^N_{y_1y}(x)[f^N(x,y,t)-f^N(x,y_1,t)]^2\right|\leq C'\rho_{N,f^N}.
\end{equation}
\end{itemize}
\end{assumption}
\noindent In Section \ref{sec_appli}, we show that a large class of conductance-based neuron models satisfies all these technical assumptions. Let us define, for technical purpose, the function,
\begin{equation}
\Psi(x)=\frac{2}{x^2}\int_0^x\log(1+y)\dd y,
\end{equation}
defined for $x\geq-1$. Note that $\Psi$ is continuous on $[0,\infty)$ with $\Psi(0)=1$ and $\lim_{x\to\infty}\Psi(x)=0$. We are now ready to state our main result.
\begin{theorem}\label{thm_main}
Under Assumptions \ref{ass_Q}, \ref{ass_mac}, \ref{ass_mac2} and \ref{Assump_for_Poi}, there exist constants $C_i$, $1\leq i\leq 5$, depending only on $T$ and a sequence $(\rho_N)$ going to infinity with $N$ such that for any $(\e,N)\in(0,1)\times\N$ and $\delta>0$:
\begin{align}
&\PP(\sup_{t\in[0,T]}\|\xen_t-\xbn_t\|^2\geq\delta)\nonumber\\
&\leq\, \PP(\|\xen_0-\xbn_0\|^2\geq C_1\delta)+ 1_{\e(\beta_N+\gamma_N)\geq C_2\delta}+C_3\exp{\left(\left(-C_4\frac{\delta^2}{\e\rho_N}\right)\Psi\left(C_5\frac{\delta\alpha_N}{\rho_N}\right)\right)}.
\end{align}
\end{theorem}
\noindent Of course, the sequence $(\rho_N)$ in the above theorem is derived from a sequence $(\rho_{N,f^N})$ (in \ref{quatb}) for appropriate functions $f^N$. Note that if $N$ is held fixed, Theorem \ref{thm_main} implies that for any time horizon $T$ and $\delta>0$, if
$$
\lim_{\e\to0}\PP(\sup_{t\in[0,T]}\|\xen_0-\xbn_0\|^2\geq\delta)=0
$$
then
$$
\lim_{\e\to0}\PP(\sup_{t\in[0,T]}\|\xen_t-\xbn_t\|^2\geq\delta)=0,
$$
such that, as mentioned above, when $N$ is held fixed, the process $\xen$ converges in law (and even in probability here) towards its averaged version $\xbn$ when $\e$ goes to zero. Moreover, this theorem allows us to give conditions under which the process $\xen$ may possess a limit when $(\e,N)$ goes to $(0,\infty)$.
\begin{corollary}
Assume that the deterministic process $\xbn$ converges towards a process $\overline{X}$ in the sense that
\begin{equation}
\forall \delta>0\quad \PP(\sup_{t\in[0,T]}\|\xbn_t-\overline{X}_t\|^2\geq\delta)\leq C_\delta\lambda_N
\end{equation}
where $(\lambda_N)$ is some positive numerical sequence which goes to $0$ when $N$ goes to infinity and $C_\delta$ is some positive constants depending on $\delta>0$. Then, under Assumptions \ref{ass_Q}, \ref{ass_mac}, \ref{ass_mac2} and \ref{Assump_for_Poi}, there exist constants $C_i$, $1\leq i\leq 5$ depending only on $T$ and a sequence $(\rho_N)$ going to infinity with $N$ such that for any $(\e,N)\in(0,1)\times\N$ and $\delta>0$:
\begin{align}
\PP(\sup_{t\in[0,T]}\|\xen_t-\overline{X}_t\|^2\geq\delta)\leq&\, \tilde{C}_\delta\lambda_N+\PP(\|\xen_0-\overline{X}_0\|^2\geq C_1\delta)\nonumber\\
&+ 1_{\e(\beta_N+\gamma_N)\geq C_2\delta}+C_3\exp{\left(-C_4\frac{\delta^2}{\e\rho_N}\Psi\left(C_5\frac{\delta\alpha_N}{\rho_N}\right)\right)},
\end{align}
where $\tilde{C}_\delta$ is some positive constants depending only on $\delta$. Therefore, if the conditions,
\begin{equation}\label{cond_conv}
\lim_{(\e,N)\to(0,\infty)}\e(\beta_N+\gamma_N)=0,\quad \lim_{(\e,N)\to(0,\infty)}\frac{1}{\e\rho_N}\Psi\left(C_5\frac{\delta\alpha_N}{\rho_N}\right)=+\infty,
\end{equation}
are fulfilled, then $\xen$ converges towards $\overline{X}$ in probability:
\begin{equation}
\forall\delta>0\quad \lim_{(\e,N)\to(0,\infty)}\PP(\sup_{t\in[0,T]}\|\xen_t-\overline{X}_t\|^2\geq\delta)=0.
\end{equation}
\end{corollary}
\noindent Since $\psi$ is positive on $[0,+\infty)$, the second part of condition (\ref{cond_conv}) is fulfilled as long as
$$
\lim_{(\e,N)\to(0,\infty)}\e\rho_N=0
$$
and the sequence $\left(\frac{\alpha_N}{\rho_N}\right)$ is bounded.

\section{Application to electro-physiology}\label{sec_appli}

\subsection{Description of conductance-based neuron models}\label{sec_mod}

We proceed to the description of spatially extended conductance-based neuron models. For a more complete description, see \cite{Au08,GT12,Hi01}. As is well known, neurons are interconnected cells which communicate between each-other using electro-chemical signals. The model we are interested in describes the generation and propagation of a nerve impulse at the level of one single neuron. The part of the neuron which is responsible for the transmission of the action potential through relatively long distances (compare to the size of the cell-body and the dendrites of the neural cell) is the axon or nerve fiber. It is often described as a cable longer than larger and this is why we model it by an unidimensional cable, the segment $I=[0,1]$. An action potential is generated at the initial portion of the axon thanks to a ion channel mechanism. Ion channels are membrane proteins which allow the exchange of ions between the internal and external cellular media. They are present in finite but large number all along the nerve fiber at loci $z_i=\frac iN\in\mathring{I}=(0,1)$ for $i\in\{1,\ldots,N-1\}$. To help to the understanding of the model, we present a schematic view of one single neural cell in Figure \ref{fig_neurone}, on which the different elements introduced in this paragraph are represented.\\
 Ion channels may take different states in a finite state space $E$ (typically a state is "to be open" or "to be closed"). We denote the state of the ion channel at position $z_i$ by $y(i)\in E$. Therefore, a possible configuration for all the ion channels is $y=(y(i))_{1\leq i\leq N}\in E_N=E^N$. When an ion channel at locus $z_i$ is in state $r(i)$, it allows a current to pass. Moreover, this current is of the form:
\[
c_{y(i)}(v_{y(i)}-x(z_i))
\]
where $c_{y(i)}\geq0$ is the conductance associated to the ion channel in state $y(i)$, $v_{y(i)}\in\R$ is the associated driven potential which tells us if the current is inward or outward and $x(z_i)$ is the local potential of the membrane at locus $z_i$. Let $H=L^2(I)$ be the state space for the membrane potential. The total current generated by the ion channels is then, for $(x,y)\in H\times N$:
\begin{equation}
F^N(x,y)=\frac{1}{N}\sum_{i=1}^{N-1}c_{y(i)}(v_{y(i)}-(x,\phi^N_i))\phi^N_i.
\end{equation}
Remark that this is a spatial averaging of all the ionic currents. The function $\phi^N_i$ specifies that the transmission of the current is localized around $z_i$. They are renormalized mollifiers, 
\[
\forall\xi\in I\quad\phi^N_i(\xi)=N\psi\left(N(\xi-z_i)\right)
\]
where $\psi$ is the mollifier
\[
\forall\xi\in I\quad\psi(\xi)=\exp\left(\frac{-1}{1-\xi^2}\right)1_{[-1,1]}(\xi).
\]
As explained above, the channel at locus $z_i$ may switch between the states of $E$. This switching depends on the local potential of the membrane. If we denote by $X_t$ the membrane potential at time $t$ and $Y_t$ the configuration of the ion channels at the same time, then
\begin{equation}\label{mod_jump}
\PP(Y_{t+h}(i)=e_2|Y_t(i)=e_1)=\frac1\e a_{e_1e_2}((X_t,\phi^N_i))h+o\left(\frac{h}{\e}\right),\quad e_1\neq e_2\in E.
\end{equation}
The non-negative functions $(a_{e_1e_2}(\cdot))_{e_1e_2}$ are the jump rates from one state of $E$ to another. They are either bounded and Lipschitz positive functions or equal to the null function. The parameter $\e\in(0,1)$ accelerate this rate of jump. Besides, the ion channels are assumed to evolve independently over infinitesimal time-scales.\\
Once a current is generated by the reaction term $F^N$, it is propagated along the nerve fiber thanks to the Laplacian operator $\D$:
\begin{equation}\label{mod_cont}
\partial_t \xen_t=\D \xen_t+F^N(\xen_t,\yen_t).
\end{equation}
This equation is endowed with zero-Dirichlet boundary conditions (corresponding to a clamped axon). To emphasize the role of $\e$ and $N$ in the process, we write from now on $(\xen,\yen)$ for the process satisfying system (\ref{mod_jump}-\ref{mod_cont}) with initial conditions $(\xen_0,\yen_0)\in H\times E_N$ (which may be random).\\
\begin{figure}
\begin{center}
\begin{tikzpicture}[scale=1]
\draw[gray,fill] (0,0)--(3,0)--(3,-0.1)--(0,-0.1)..controls +(-0.5,0) and +(-0.1,0.75)..(-1,-1)..controls +(-0.25,0) and +(0.25,0)..(-1.1,-1)..controls +(0,0.5) and +(0.25,0.25)..(-2,-0.5)--(-2.6,-1)--(-2.1,-0.45)..controls +(0.25,0.25) and +(0.25,-0.25)..(-2,0.5)--(-1.9,0.6)..controls +(0.5,-0.5) and +(-0.25,-0.25)..(-0.6,0.7)--(-0.5,0.6)..controls +(-0.5,-0.25) and +(-0.75,0)..(0,0);
\draw[<->] (0,-0.7)-- node[midway,below]{axon: $I$} (5,-0.7);
\draw (0.1,-0.1) circle (0.05) ;
\draw (0.3,-0.1) circle (0.05) ;
\draw (0.6,-0.1) circle (0.05) ;
\draw (1.5,-0.1) circle (0.05) ;
\draw (0,0.3)..controls +(0.15,0) and +(-0.15,0)..(0.2,0.3)..controls +(0,0.15) and +(-0.15,0)..(0.3,1.3)..controls +(0,0.15) and +(-0.15,-0.15)..(0.5,0.2)..controls +(0.15,0.15) and +(-0.15,0)..(0.7,0.3);
\draw[->] (0.5,0.75)--node[midway,above]{\small $\xen_t$}(1.3,0.75);
\draw[gray,fill] (-2.4,-0.6) circle (0.07);
\draw[gray,very thick] (-2.4,-0.6)--(-2.3,-0.7);
\draw[black,fill] (-2.55,-0.45) circle (0.07);
\draw[black,line width=2.5] (-2.55,-0.45)..controls +(-0.25,0.25) and +(1,0)..(-3.5,0.2)..controls +(-0.25,0) and +(1,0)..(-4,0.6)..controls +(-0.25,0) and +(1,0)..(-4.5,0.8);
\draw[gray,very thick,dotted] (3,0)--(4.5,0);
\draw[gray,very thick,dotted] (3,-0.1)--(4.5,-0.1);
\draw[gray,fill] (4.5,0)--(5,0)..controls +(0.5,0) and +(-0.5,0)..(6.5,1)--(6.5,0.9)..controls +(-0.5,0) and +(0.5,0)..(5.2,-0.05)..controls +(0.5,0) and +(-0.5,0)..(6.7,-0.6)--(6.7,-0.7)..controls +(-0.5,0) and +(0.5,0)..(5,-0.1)--(4.5,-0.1)--(4.5,0);
\draw[gray,fill] (6.55,0.95) circle(0.08);
\draw[gray,fill] (6.7,-0.65) circle(0.08);
\draw[black,fill] (6.7,1.4)..controls +(0.25,-0.5) and +(-0.5,0)..(7.5,0.6)--(7.5,0.7)..controls +(-0.5,0) and +(0.3,-0.5)..(6.7,1.4);
\draw (1.5,-0.1)--(1.8,0.5)--(2.1,0.5) node[right]{ionic channel: {\small $\yen_t(i)$}};
\draw (-2.5,-0.7)--(-3,-1.2)--(-3.25,-1.2) node[left]{synapse};
\draw (6.55,0.8)--(7,-0.05)--(7.2,-0.05) node[right]{synapses};
\draw (6.7,-0.5)--(7,-0.05);
\draw[black,fill] (6.75,1) circle (0.07);
\draw[black,very thick] (6.75,1)--(6.9,1.075);
\draw (-0.8,-0.5)--(-0.1,-1.3)--(0.2,-1.3) node[right]{soma};
\draw (-0.7,0.7)--(-1,1.4)--(-1.3,1.4) node[left]{dendrite};
\end{tikzpicture}
\end{center}
\caption{Schematic view of one single neuron with the soma (cell body), the dendrites, synapses  and axon, see \cite{Hi01} for the precise definition of these terms. An action potential is generated at the initial portion of the axon thanks to the presence of the ionic channels, depicted by black circles here, which can open and close at voltage dependent rate, allowing a flow of ions to enter or leave the cell, creating a current. The state of the ionic channel at the locus $z_i$ at time $t$ is $\yen_t(i)$ and follows the dynamic of equation (\ref{mod_jump}). The created action potential propagates along the nerve fiber, assimilated to the segment $I$, according to the dynamic of $\xen_t$ given by equation (\ref{mod_cont}).}\label{fig_neurone}
\end{figure}
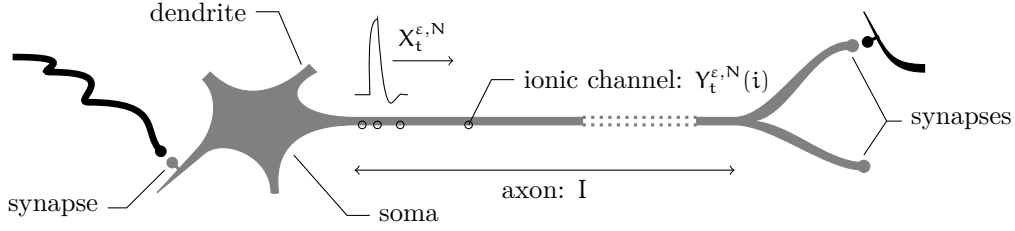
Concerning averaging when $N$ is held fixed, such a system have been studied in \cite{GT12,GT14}. For a law of large number when $N$ goes to infinity but with $\e$ held fixed, see \cite{RTW12,TR13}. The equations (\ref{mod_cont}) and (\ref{mod_jump}) describe a so-called conductance-based neuron model. Such models are able to reproduce numerous features of biological neurons such as their excitability, sensitivity and robustness, see for example \cite{TSS20}.\\
We assume that for any $\zeta\in\R$, the continuous time Markov chains with jump rates $(a_{e_1e_2}(\zeta))_{e_1e_2}$ has a unique invariant measure $\nu(\zeta)$ which is bounded and Lipschitz in its argument $\zeta$. Therefore, the averaging measure $\mu_N$ is defined, for any potential $x\in H$, by:
\begin{displaymath}
\mu_N(x)=\otimes_{i=1}^{N-1}\nu((x,\phi^N_i)).
\end{displaymath}

\subsection{Spatio-temporal averaging for these models}

The averaged reaction term for the process satisfying equations (\ref{mod_jump}-\ref{mod_cont}) is given, for $x\in H$ and $N\in\N$ by:
\begin{align*}
\overline{F}^N(x)&=\int_{E_N}F^N(x,y)\mu(x)(\dd y)\\
&=\frac{1}{N}\sum_{i=1}^{N-1}\sum_{e\in E}\nu((x,\phi^N_i))(\{e\})c_{e}(v_{e}-(x,\phi^N_i))\phi^N_i.
\end{align*}

\noindent In Section \ref{sec_proof_neuron}, we show that the conductance based neuron model of Section \ref{sec_mod} satisfies Assumption \ref{Assump_for_Poi} with $\alpha_N=N$, $\beta_N=N\sqrt{N}$, $\gamma_N=N^3$ and $\rho_N=N^3$. Besides, this is rather easy to see that the model also satisfies Assumptions \ref{ass_Q}, \ref{ass_mac} and \ref{ass_mac2}. In the context of conductance-based neuron models, Theorem \ref{thm_main} reads as follows.

\begin{theorem}\label{thm_main_mod}
There exist constants $C_i$, $1\leq i\leq 5$, depending only on $T$ such that for any $(\e,N)\in(0,1)\times\N$ and $\delta>0$:
\begin{align}
\PP(\sup_{t\in[0,T]}\|\xen_t-\xbn_t\|^2\geq\delta)&\leq \PP(\|\xen_0-\xbn_0\|^2\geq C_1\delta)+1_{\e(N\sqrt{N}+N^3)\geq C_2\delta}\nonumber\\
&+C_3\exp{\left(-C_4\frac{\delta^2}{\e N^3}\Psi\left(C_5\frac{\delta}{N^2}\right)\right)}.
\end{align}
\end{theorem}
\noindent Since $\lim_{N\to\infty}\Psi\left(C_5\frac{\delta}{N^2}\right)=\Psi(0)=1$, this shows that for the limit to be zero when $\e$ and $N$ converge jointly to $0$ and $\infty$, it is sufficient to impose the relative speed of convergence:
\[
\lim_{(\e,N)\to (0,\infty)} \e N^3=0.
\]
This means that $\xen$ converges when the jump frequency is somehow high with respect to the size of the population of ion channels ($\e={\rm o}\left(\frac{1}{N^3}\right)$). As far as we know, this fact has not been noticed before. When $N$ goes to infinity, the process $\xbn$ possesses a limit $\overline{X}$ satisfying the following PDE,
$$
\partial_t \overline{X}_t=\D \overline{X}_t+\sum_{e\in E}\nu(\overline{X}_t)(\{e\})c_{e}(v_e-\overline{X}_t).
$$
One may show, but this is not our purpose here, that this limit takes place in $\mathcal{C}([0,T],L^2(I))$. Therefore, $\xen$ converges towards $\overline{X}$ in probability once $\e N^3$ goes to zero.

\section{Proofs}\label{sec_proof}

This section is devoted to the proofs of Theorems \ref{thm_main} and \ref{thm_main_mod}.

\subsection{Proof of Theorem \ref{thm_main}}

\noindent We proceed to the proof of Theorem \ref{thm_main}. The plan of the campaign is the following:\\
\noindent \textbf{Step 1:} Show that the following Gronwall-like inequality almost-surely holds for $\|\xen_t-\xbn_t\|^2$:
\begin{align}\label{pl_eq}
\|\xen_t-\xbn_t\|^2&\leq\|\xen_0-\xbn_0\|^2+2C_1\int_0^t\|\xen_s-\xbn_s\|^2\dd s\\
&+ 2\int_0^t(F^N(\xen_s,\yen_s)-\overline{F}^N(\xen_s),\xen_s-\xbn_s)\dd s.\nonumber
\end{align}
where $C_1$ is some constant.\\
\textbf{Step 2:} Find the semi-martingale decomposition of the term
$$2\int_0^t(F^N(\xen_s,\yen_s)-\overline{F}^N(\xen_s),\xen_s-\xbn_s)\dd s.$$
For this purpose, use the Poisson equation associated to the jump operator $Q^N$.\\
\textbf{Step 3:} Bound the finite variation part thanks to almost-sure estimates.\\
\textbf{Step 4:} Plug this latter estimate in (\ref{pl_eq}) and use Gronwall's lemma to almost-surely bound $\|\xen_t-\xbn_t\|^2$:
\begin{align*}
\sup_{t\in[0,T]}\|\xen_t-\xbn_t\|^2&\leq\left(\|\xen_0-\xbn_0\|^2+{\rm O}(\e(\beta_N+\gamma_N))+2\e\sup_{t\in[0,T]}|M^{\e,N}_t|\right)e^{C_1T},
\end{align*}
$\PP$-a.s.\\
\textbf{Step 5:} Bound the remaining martingale part $\sup_{t\in[0,T]}|M^{\e,N}_t|$ in probability using a martingale exponential inequality.\\
\textbf{Step 6:} Aggregate all the estimates to conclude.\\

\noindent \textbf{Step 1: Gronwall-like inequality.} Let us consider the following decomposition:
\begin{align}
&\ddt\|\xen_t-\xbn_t\|^2\\
&=2<A(\xen_t-\xbn_t),\xen_t-\xbn_t>+2(F^N(\xen_t,\yen_t)-\overline{F}^N(\xbn_t),\xen_t-\xbn_t)\nonumber\\
&\leq 2([\overline{F}]_d-[A]_d)\|\xen_t-\xbn_t\|^2+2(F^N(\xen_t,\yen_t)-\overline{F}^N(\xen_t),\xen_t-\xbn_t).\label{pre_gron}
\end{align}
We denote by $C_1$ the constant $[\overline{F}]_d-[A]_d$. From (\ref{pre_gron}) we almost-surely obtain:
\begin{align}\label{gron}
\|\xen_t-\xbn_t\|^2&\leq\|\xen_0-\xbn_0\|^2+2C_1\int_0^t\|\xen_s-\xbn_s\|^2\dd s\\
&+ 2\int_0^t(F^N(\xen_s,\yen_s)-\overline{F}^N(\xen_s),\xen_s-\xbn_s)\dd s.\nonumber
\end{align}
\textbf{Step 2: Semi-martingale decomposition.} Our aim is to obtain a semi-martingale decomposition for the term defined by:
\begin{equation}
V^{\e,N}_t=\int_0^t(F^N(\xen_s,\yen_s)-\overline{F}^N(\xen_s),\xen_s-\xbn_s)\dd s.
\end{equation}
\begin{proposition}[Poisson equation]\label{prop_Poi}
For any $N\in\N$, there is a unique solution $f^N$ defined on $H\times E_N\times \R_+$ of the following Poisson equation:
\begin{equation}\label{p_eq}
\left\{\begin{array}{rcl}
Q^N(x)f^N(x,y,t)&=&(F^N(x,y)-\overline{F}^N(x),x-\xbn_t),\\
\int_{E_N} f^N(x,y,t)\mu_N(x)(\dd y)&=&0.
\end{array}
\right.
\end{equation}
Moreover, $f^N$ satisfies the following properties:
\begin{enumerate}
\item $f^N\in \mathcal{D}(\mathcal{A}^{\e,N})$;
\item $\sup_{t\in[0,T],x\in B[0,[X]_b],y\in E_N}|f^N(x,y,t)|={\rm O}(\alpha_N)$;
\item the map $f^N$is continuously Fr\'echet differentiable with respect to $x$ and has a representation $f^N_x$ in H satisfying $\sup_{t\in[0,T],x\in B[0,[X]_b],y\in E_N}\|f^N_x(x,y,t)\|={\rm O}(\beta_N)$;
\item the map $f^N$is continuously differentiable with respect to $t$ and this derivative satisfies $\sup_{t\in[0,T],x\in B[0,[X]_b],y\in E_N}|f^N_t(x,y,t)|={\rm O}(\gamma_N)$.
\end{enumerate}
\end{proposition}
\begin{proof}
Let us fix $N\in\N$, $x\in H$ and $t\in\R_+$ in (\ref{p_eq}). Then Equation (\ref{p_eq}) has a solution if, and only if, the map $\Phi_{N,x,t}:y\in E\mapsto (F^N(x,y)-\overline{F}^N(x),x-\xbn_t)$ is in ${\rm Im}Q^N(x)$. According to the Fredholm alternative in finite dimension, $\Phi_{N,x,t}$ is in ${\rm Im}Q^N(x)$ if, and only if, $\int_E\Phi_{N,x,t}(y)\mu_N(x)(\dd y)=0$. The latter equality is true by definition of $\overline{F}^N$. Thus there exists $f^N(x,\cdot,t)$ satisfying
\begin{equation}\label{eq_poi_i}
Q^N(x)f^N(x,y,t)=(F^N(x,y)-\overline{F}^N(x),x-\xbn_t).
\end{equation}
One may always choose $f^N$ such that $\int_{E_N} f^N(x,y,t)\mu_N(x)(\dd y)=0$ by projection and therefore ensure uniqueness. According to the equality (\ref{eq_poi_i}), the fact that $f^N$ is in $\mathcal{D}(\mathcal{A}^{\e,N})$ and satisfies points 2. 3. and 4. is a direct consequence of Assumption \ref{Assump_for_Poi} and the fact that for each $x\in H$, $Q^N(x)$ is a finite dimensional operator.
\end{proof}

\begin{proposition}
There exists a squared integrable martingale $M^{\e,N}$ such that for all $t\in[0,T]$:
\begin{equation}\label{def_mart}
f^N(\xen_t,\yen_t,t)=f^N(\xen_0,\yen_0,0)+\int_0^t \mathcal{A}^{\e,N}f^N(\xen_s,\yen_s,s)\dd s +M^{\e,N}_t.
\end{equation}
Its bracket satisfies:
\begin{equation}\label{def_brack}
<M^{\e,N}>_t=\frac{1}{\e}\sum_{y\in E_N}\int_0^t Q^N_{\yen_sy}(\xen_s)[f^N(\xen_s,y,s)-f^N(\xen_s,\yen_s,s)]^2\dd s.
\end{equation}
\end{proposition}
\begin{proof}
The fact that $M^{\e,N}$ is a martingale follows from the fact that $f^N\in\mathcal{D}(\mathcal{A}^{\e,N})$. From the properties of $f^N$, see Proposition \ref{prop_Poi}, we see that this martingale is squared integrable. The expression for the bracket follows from the identification of the finite variation  part in the semi-martingale expansion of $(f^N)^2$ using on the one hand the Dynkin formula and on the other hand the It\^o formula.  
\end{proof}
\noindent For convenience, we introduce the notation $\partial_{i,s}$, $i=1,3$, denoting the partial derivative with respect to $s$ \textit{along the flow}. More explicitly, in our context,
$$
\partial_{1,s}f^N(\xen_s,\yen_s,s)=\lim_{h\to0}\frac1h(f^N(\xen_{s+h},\yen_s,s)-f^N(\xen_s,\yen_s,s))
$$
and
$$
\partial_{3,s}f^N(\xen_s,\yen_s,s)=\lim_{h\to0}\frac1h(f^N(\xen_{s},\yen_s,s+h)-f^N(\xen_s,\yen_s,s)).
$$
The existence of such derivatives is ensured by Proposition \ref{prop_Poi} and an exact representation for $\partial_{1,s}$ is given in the proof of Proposition \ref{prop_vf_borne}.
\begin{proposition}\label{prop_V_decompo}
The process $(V^{\e,N}_t,t\geq0)$ has the following semi-martingale decomposition, $\PP$-a.s.,
\begin{align*}
V^{\e,N}_t&=\e f^N(\xen_t,\yen_t,t)-\e f^N(\xen_0,\yen_0,0)-\e\int_0^t(\partial_{1,s}+\partial_{3,s})f^N(\xen_s,\yen_s,s)\dd s-\e M^{\e,N}_t.
\end{align*}
\end{proposition}
\begin{proof}
This is simply a re-arrangement of (\ref{def_mart}) using the definition of $V^{\e,N}$ and Proposition \ref{prop_Poi}.
\end{proof}
\noindent \textbf{Step 3: A bound for the finite variation part.} From now on, our aim is to bound in probability the different terms of the semi-martingale decomposition of $(V^{\e,N}_t,t\geq0)$. We begin with the finite variation part.
\begin{proposition}\label{prop_vf_borne}
We have, 
\begin{equation}
\sup_{t\in[0,T]}\left|\int_0^t(\partial_{1,s}+\partial_{3,s})f^N(\xen_s,\yen_s,s)\dd s\right|={\rm O}(\beta_N+\gamma_N),
\end{equation}
$\PP$-a.s. where the big $O$ is uniform in $\e$.
\end{proposition}
\begin{proof}
Since $f^N$ is in $\mathcal{D}(\mathcal{A}^{\e,N})$ and $f^N_x(\cdot,y,t)\in H^*$ may be interpreted as an element of $H$, from the chain rules ( see \cite{BR11}, Theorem 4, item iii.) we have:
\[
\partial_{1,s}f^N(\xen_s,\yen_s,s)=<A\xen_s+F^N(\xen_s,\yen_s),f^N_x(\xen_s,\yen_s,s)>.
\]
Thus, from Assumption (\ref{A_d}), (\ref{F_d}) and Proposition \ref{prop_bound}, we obtain that there exists a constant $C$ independent of $\e$ and $N$ such that:
\[
\sup_{s\in[0,T]}\|A\xen_s+F^N(\xen_s,\yen_s)\|_*\leq C,
\]
$\PP$-a.s. Then, the fact that 
\[
\sup_{t\in[0,T]}\left|\int_0^t\partial_{1,s}f^N(\xen_s,\yen_s,s)\dd s\right|={\rm O}(\beta_N+\gamma_N)
\]
follows from Proposition \ref{prop_Poi}. The proof is quite similar for the $\partial_{3,s}$-derivative.
\end{proof}
\noindent \textbf{Step 4: Gronwall's lemma.} From Propositions \ref{prop_Poi}, \ref{prop_V_decompo} and \ref{prop_vf_borne} we deduce that,
\begin{equation}
\sup_{t\in[0,T]}|V^{\e,N}_t|\leq {\rm O}(\e(\beta_N+\gamma_N))+\e\sup_{t\in[0,T]}|M^{\e,N}_t|,
\end{equation}
$\PP$-a.s. Therefore,
\begin{align}\label{gron_bis}
\|\xen_t-\xbn_t\|^2\leq&~\|\xen_0-\xbn_0\|^2+2C_1\int_0^t\|\xen_s-\xbn_s\|^2\dd s\\
&+{\rm O}(\e(\beta_N+\gamma_N))+2\e\sup_{t\in[0,T]}|M^{\e,N}_t|,
\end{align}
$\PP$-a.s. Using Gronwall's lemma we obtain, $\PP$-a.s.,
\begin{align}\label{gron_bis}
\sup_{t\in[0,T]}\|\xen_t-\xbn_t\|^2&\leq\left(\|\xen_0-\xbn_0\|^2+{\rm O}(\e(\beta_N+\gamma_N))+2\e\sup_{t\in[0,T]}|M^{\e,N}_t|\right)e^{C_1T},
\end{align}
\textbf{Step 5: A bound for the martingale part.} In order to obtain the bound in probability for the martingale part, we will use an exponential inequality.
\begin{proposition}
There exist two constants $C,\hat{C}$ depending only on $T$ and a sequence $(\rho_N)$ going to infinity with $N$ such that:
\begin{equation}
\forall\delta>0\quad \PP(2\e e^{C_1T}\sup_{t\in[0,T]}|M^{\e,N}_t|\geq\delta)\leq2\exp\left(-\frac{e^{-2C_1 t}}{8C}\frac{\delta^2}{\e\rho_N}\Psi\left(\frac{\hat{C} e^{-C_1T}}{2C}\frac{\delta\alpha_N}{\rho_N}\right)\right),
\end{equation}
\end{proposition}
\begin{proof}
From Assumptions \ref{ass_Q} and \ref{Assump_for_Poi}, the fact that $\sup_{t\in[0,T],x\in B[0,[X]_b],y\in E_N}|f^N(x,y,t)|={\rm O}(\alpha_N)$ and the expression of the bracket (\ref{def_brack}), we see that there exists a constant $C$ such that:
\[
\sup_{t\in[0,T]}<M^{\e,N}>_t\leq C\frac{\rho_N}{\e},
\]
$\PP$-a.s with $\rho_N=\rho_{N,f^N}$ and from the semi-martingale expansion given in Proposition \ref{prop_V_decompo} there exists a constant $\hat{C}$ such that
$$
\sup_{t\in[0,T]}|\Delta M^{\e,N}_t|\leq \hat{C}\alpha_N,
$$
$\PP$-a.s where, $\Delta M^{\e,N}_t=M^{\e,N}_t- M^{\e,N}_{t^-}$. Then, using a martingale exponential inequality, see \cite{DV01}, we obtain, for any $\delta>0$:
\begin{align*}
\PP(2\e e^{C_1T}\sup_{t\in[0,T]}|M^{\e,N}_t|\geq\delta)&=\PP\left(2\e e^{C_1T}\sup_{t\in[0,T]}|M^{\e,N}_t|\geq\delta;\sup_{t\in[0,T]}<M^{\e,N}>_t\leq C\frac{\alpha^2_N\eta_N}{\e}\right)\\
&=\PP\left(\sup_{t\in[0,T]}|M^{\e,N}_t|\geq\frac{\delta e^{-C_1T}}{2\e};\sup_{t\in[0,T]}<M^{\e,N}>_t\leq C\frac{\rho_N}{\e}\right)\\
&\leq 2\exp\left(-\frac12\frac{\left(\frac{\delta e^{-C_1T}}{2\e}\right)^2}{C\frac{\rho_N}{\e}}\Psi\left(\frac{\hat{C}\alpha_N\frac{\delta e^{-C_1T}}{2\e}}{C\frac{\rho_N}{\e}}\right)\right)\\
&=2\exp\left(-\frac{e^{-2C_1 t}}{8C}\frac{\delta^2}{\e\rho_N}\Psi\left(\frac{\hat{C} e^{-C_1T}}{2C}\frac{\delta\alpha_N}{\rho_N}\right)\right).
\end{align*}
\end{proof}
\noindent \textbf{Step 6: Conclusion.} According to (\ref{gron_bis}), there exists a constant $\tilde C$ such that:
\begin{align}
\sup_{t\in[0,T]}\|\xen_t-\xbn_t\|^2&\leq\left(\|\xen_0-\xbn_0\|^2+\tilde C\e(\beta_N+\gamma_N)+2\e\sup_{t\in[0,T]}|M^{\e,N}_t|\right)e^{C_1T},
\end{align}
$\PP$-a.s. Let $\delta>0$, we have:
\begin{align*}
\PP(\sup_{t\in[0,T]}\|\xen_t-\xbn_t\|^2\geq\delta)\leq&\,\, \PP(e^{C_1T}\|\xen_0-\xbn_0\|^2\geq\frac\delta3)+\PP(e^{C_1T}C\e(\beta_N+\gamma_N)\geq\frac\delta3)\\
&+\,\PP(e^{C_1T}2\e\sup_{t\in[0,T]}|M^{\e,N}_t|\geq\frac\delta3)\\
\leq&\,\, \PP(\|\xen_0-\xbn_0\|^2\geq\frac{\delta e^{-C_1T}}{3})+1_{\e(\beta_N+\gamma_N)\geq\frac{e^{-C_1T}\delta}{3C}}\\
&+\,2\exp\left(-\frac{e^{-2C_1 T}}{72C}\frac{\delta^2}{\e\rho_N}\psi\left(\frac{\hat{C} e^{-C_1T}}{6C}\frac{\delta\alpha_N}{\rho_N}\right)\right).
\end{align*}
This concludes the proof.

\subsection{Proof of Theorem \ref{thm_main_mod}}\label{sec_proof_neuron}

In this part, we show that the process defined in Section \ref{sec_mod} satisfies the assumptions needed to obtain Theorem \ref{thm_main}. In the sequel, for convenience, we will use the following notations:
$$
c_+=\max_{e\in E}c_e,\,\,v_+=\max_{e\in E}|v_e|,\,\,a_+=\sup_{e_1,e_2\in E,\zeta\in\R}a_{e_1e_2}(\zeta),\,\, \mu_+=\sup_{\zeta\in\R,e\in E}\nu(\zeta)(\{e\}).
$$
Note also the elementary facts that, for any $i\in 1,\ldots,N-1$, $\|\phi^N_i\|\leq\sqrt{2N}$ and that $\|\psi^{''}\|<\infty$. We more precisely show that Assumption \ref{Assump_for_Poi} is satisfied. The proof that Assumptions \ref{ass_Q}, \ref{ass_mac} and \ref{ass_mac2} are also satisfied is not difficult and is thus left to the reader.
\begin{proposition}
Let us consider the function $\Phi^N:(x,y,z)\in H\times E^N\times H\mapsto (F^N(x,y)-\overline{F}^N(x),x-z)$. For any $C>0$, one has
\begin{enumerate}
\item $\Phi^N$ satisfies: $\sup_{x,z\in B[0,C],y\in E_N}|\Phi^N(x,y,z)|={\rm O}(N).$
\item $\Phi^N$ is continuously Fr\'echet differentiable with respect to its first variable and its Fr\'echet derivative which may be identified with an element of $L^2(I)$ denoted $\Phi^N_x$ satisfying:
\[
\sup_{x,z\in B[0,C],y\in E_N}|\Phi^N_x(x,y,z)|={\rm O}(N\sqrt{N}).
\]
\item The map $t\mapsto \Phi^N(x,y,\xbn_t)$ is continuously differentiable  with respect to $t$ and satisfies
\[
\sup_{x,z\in B[0,C],y\in E_N}|\Phi^N_t(x,y,z)|={\rm O}(N^3).
\]
\item $\Phi^N$ is in $\mathcal{D}(\mathcal{A}^{\e,N})$.
\end{enumerate}
\end{proposition}
\begin{proof}
Let $N\in\N$, $y\in E_N$ and $x,z\in B[0,C]$. A direct calculation using the fact that $\|\phi^N_i\|\leq\sqrt{2N}$ leads to
\begin{align*}
|(F^N(x,y)-\overline{F}^N(x),x-z)|&=\left|\frac{1}{N}\sum_{i=1}^{N-1}\sum_{e\in E}c_{e}(v_{e}-(x,\phi^N_i))(1_{\{e\}}(y(i))-\nu((x,\phi_i))(\{e\}))(\phi^N_i,x-z)\right|\\
&\leq c_+(v_++\|x\|\sqrt{2N})(1+\mu_+)|E|\|x-z\|\sqrt{2N}\\
&\leq c_+(v_++\sqrt{2N}C)(1+\mu_+)|E|2C \sqrt{2N}.
\end{align*}
This ensures point 1. Then we compute the Fr\'echet derivative of $\Phi^N$ with respect to its first variable.
\begin{align*}
\frac{\dd\Phi^N}{\dd x}(x,y,z)[h]=&\,-\frac{1}{N}\sum_{i=1}^{N-1}\sum_{e\in E}c_{e}(h,\phi^N_i)(1_{\{e\}}(y(i))-\nu((x,\phi_i))(\{e\}))(\phi^N_i,x-z)\\
&-\frac{1}{N}\sum_{i=1}^{N-1}\sum_{e\in E}c_{e}(v_{e}-(x,\phi^N_i))(h,\phi^N_i)\nu'((x,\phi_i))(\{e\})(\phi^N_i,x-z)\\
&+\frac{1}{N}\sum_{i=1}^{N-1}\sum_{e\in E}c_{e}(v_{e}-(x,\phi^N_i))(1_{\{e\}}(y(i))-\nu((x,\phi_i))(\{e\}))(\phi^N_i,h).
\end{align*}
Therefore, $\frac{\dd\Phi^N}{\dd x}(x,y,z)$, which is by definition an element of $L^2(I)^*$, may be identified with the element of $L^2(I)$:
\begin{align*}
\frac{\dd\Phi^N}{\dd x}(x,y,z)[h]=&\,-\frac{1}{N}\sum_{i=1}^{N-1}\sum_{e\in E}c_{e}(1_{\{e\}}(y(i))-\nu((x,\phi_i))(\{e\}))(\phi^N_i,x-z)\phi^N_i\\
&-\frac{1}{N}\sum_{i=1}^{N-1}\sum_{e\in E}c_{e}(v_{e}-(x,\phi^N_i))\nu'((x,\phi^N_i))(\{e\})(\phi^N_i,x-z)\phi^N_i\\
&+\frac{1}{N}\sum_{i=1}^{N-1}\sum_{e\in E}c_{e}(v_{e}-(x,\phi^N_i))(1_{\{e\}}(y(i))-\nu((x,\phi_i))(\{e\}))\phi^N_i.
\end{align*}
For $x,z\in B[0,C]$ and $y\in E_N$, the three above terms are bounded almost-surely in $L^2(I)$ by:
\[
c_+(1+\mu_+)|E|2C\sqrt{2N}\sqrt{2N}+c_+(v_++C\sqrt{2N})|E|\mu'_+2C\sqrt{2N}\sqrt{2N}+c_+(v_++C\sqrt{2N})(1+\mu_+)|E|\sqrt{2N}.
\] 
This ensures point 2. We go on with the derivative of the map $t\mapsto \Phi^N(x,y,\xbn_t)$ with $(x,y)\in H\times E_N$ held fixed. We have,
\begin{align*}
\frac{\dd\Phi^N}{\dd t}(x,y,\xbn_t)&=<F^N(x,y)-\overline{F}^N(x),\partial_t\xbn_t>\\
&=\frac{1}{N}\sum_{i=1}^{N-1}\sum_{e\in E}c_{e}(v_{e}-(x,\phi^N_i))(1_{\{e\}}(y(i))-\nu((x,\phi_i))(\{e\}))<\phi^N_i,\Delta\xbn_t+\overline{F}^N(\xbn_t)>.
\end{align*}
Let us remark that:
\begin{align*}
<\phi^N_i,\Delta\xbn_t>&=(\Delta \phi^N_i,\xbn_t)= N^3(\psi''\left(N(\cdot-z_i)\right),\xbn_t)\leq N^2\sqrt{N}\|\psi^{''}\|\|\xbn_t\|
\end{align*}
and
\begin{align*}
<\phi^N_i,\overline{F}^N(\xbn_t)>&=\frac{1}{N}\sum_{j=1}^{N-1}\sum_{e\in E}\nu((x,\phi^N_j))(\{e\})c_{e}(v_{e}-(x,\phi_j))(\phi^N_i,\phi^N_j)\\
&=\frac{1}{N}\mu_+c_+(v_++\|x\|\sqrt{2N})|E|\|\phi^N_i\|^2\\
&\leq \mu_+c_+(v_++C\sqrt{2N})|E|2N.
\end{align*}
This ensures point 3. The map $(x,y,t)\mapsto \Phi^N(x,y,\xbn_t)$ is thus a bounded measurable function on $H\times E_N\times [0,T]$ which belongs to $\mathcal{D}(\mathcal{A}^{\e,N})$.
\end{proof}
\noindent Therefore, according to the notations of Section \ref{sec_res}, we can choose $\alpha_N=N$, $\beta_N=N\sqrt{N}$ and $\gamma_N=N^3$. The next proposition gives the form of $\rho_{N,f^N}$ explicitly.
\begin{proposition}
For any $N\in\N$, $y\in E_N$, $x\in H$ and $t\in[0,T]$, if $f^N(x,y,t)$ is bounded on $B[0,C]\times E_N\times[0,T]$, then 
\[
\sum_{i=1}^{N-1}\sum_{e\in E}[f^N(x,y_{y(i)\to e},t)-f^N(x,y,t)]^2 a_{y(i)e}((x,\phi^N_i))\leq 4 a_+|E| \rho_{N,f^N},
\]
where $\rho_{N,f^N}=N\sup_{t\in[0,T],x\in B[0,C],y\in E_N}|f^N(x,y,t)|^2$.
\end{proposition}

\noindent \textbf{Acknowledgment}. I would like to thank Vincent Renault for his kind advice on an earlier version of the manuscript, and also the two anonymous reviewers for their careful reading of our manuscript and for their constructive suggestions. 

\bibliographystyle{plain}

\end{document}